\documentclass[reqno,intlimits,twoside]{article}

\usepackage{fancyhdr}
\usepackage{amsmath}
\usepackage{amssymb}
\usepackage{cite}
\usepackage{amsthm}
\usepackage{color}

\newtheorem{lemma}{Lemma}[section]
\newtheorem{theorem}{Theorem}[section]
\newtheorem*{corollary}{Corollary}
\newtheorem*{Proposition}{Proposition}
\theoremstyle{definition}

\newtheorem{remark}{Remark}[section]

\numberwithin{equation}{section}

\begin{document}

\thispagestyle{empty}

\begin{center}
    \mbox{\hskip-0.3cm \fontsize{12}{14pt}\selectfont
    \textbf{Pitt inequality for the linear structurally damped $\sigma$-evolution equations}} \\[3mm]
\end{center}

\medskip

\ \!\!\!\hrulefill \
\begin{center}
\textbf{\fontsize{11}{14pt}\selectfont Khaldi Sa\"{i}d}
\end{center}
\vskip+1cm
\thispagestyle{empty}

\noindent \textbf{Abstract.}
This work is devoted to improve the time decay estimates for the solution and some of its derivatives of the linear structurally damped $\sigma$-evolution equations. The Pitt inequality is the main tool provided that the initial data lies in some weighted spaces.

\vskip+0.5cm


\medskip


\vskip+0.4cm


 \newpage
\pagestyle{fancy}
\fancyhead{}
\fancyhead[EC]{\hfill \textsf{\footnotesize Khaldi Sa\"{\i}d}}
\fancyhead[EL,OR]{\thepage}
\fancyhead[OC]{\textsf{\footnotesize Pitt inequality for the linear structurally damped $\sigma$-evolution equations}\hfill }
\fancyfoot[C]{}
\renewcommand\headrulewidth{0.5pt}

\section{Introduction and main results}
Our interest in this work is to derive a new time decay estimates for the solution and some of its derivatives to the Cauchy problem for homogeneous linear $\sigma$-evolution equations with structural damping:
\begin{equation}\label{1.1}
\partial^{2}_{t}u+(-\Delta)^{\sigma} u+(-\Delta)^{\delta}\partial_{t}u=0, \ u(0,x)=u_{0}(x), \ \ \partial_{t}u(0,x)=u_{1}(x),
\end{equation} 
where
$$
t\geqslant0, \ \ x\in\mathbb{R}^{n},\ \  \sigma\geq 1 \ \ and\ \  \delta\in\left( 0,\frac{\sigma}{2}\right) .
$$
First, for the ease of reading, we will use the following notations:
\begin{itemize}
	\item For any $x\in \mathbb{R}^{n}$, by $|x|$ we mean the usual norm in $\mathbb{R}^{n}$, i.e,
	$|x|=(x_{1}^{2}+\cdots x_{n}^{2})^{1/2}$, and so on for $\xi\in\mathbb{R}^{n}.$
	\item The fractional Laplacian operator $(-\Delta)^{a}$  is defined on $\mathbb{R}^{n}$ via the Fourier transform $\mathcal{F}$ as follows:
	\begin{equation}\label{1.2}
	\mathcal{F}\left((-\Delta)^{a}f\right)(\xi)= |\xi|^{2a}\mathcal{F}\left(f\right) (\xi)  \quad\text{ for all }\ \xi \in\mathbb{R}^{n}, \ \ a=\sigma, \delta.
	\end{equation}
	\item We denote $f\lesssim g$ when there exists a constant $C>0$ such that $f \leq Cg$.
	\item The notation $f\approx g$ means there exist $c_{1}, c_{2} > 0$ such that $c_{1}g\leq f \leq c_{2}g$.  
	\item For any $s\geq 0$ the Sobolev spaces $H^{s}(\mathbb{R}^{n})$ are given by:
	$$f\in H^{s}(\mathbb{R}^{n}) \ \ \text{iff}\ \  \|f\|_{H^{s}(\mathbb{R}^{n})}=\|(1+|\xi|^{2})^{\frac{s}{2}}\mathcal{F}(f)\|_{L^{2}(\mathbb{R}^{n})}<\infty.$$
	\item For any $m\in(1,2]$ and $\rho\geq 0$, the weighted Lebesgue spaces $L^{\rho,m}(\mathbb{R}^{n})$ are defined by: 
	$$ g \in L^{\rho,m}(\mathbb{R}^{n}) \ \ \text{iff} \ \ \||x|^{\rho}g\|_{L^{m}(\mathbb{R}^{n})}=\left(\int_{\mathbb{R}^{n}}|x|^{\rho m}|g(x)|^{m}dx\right)^{1/m}<\infty.$$
	
\end{itemize}
It is now known that the Cauchy problem (\ref{1.1}) has been widely studied, where the authors have derived some sharp linear estimates by assuming initial data in usual Sobolev or Lebesgue spaces (see for instance \cite{PhamKainaneReissig, D'AbbiccoEbert2017, D'AbbiccoEbertapplication, D'AbbiccoEbertdiffusion, D'AbbiccoReissig2014, DaoReissigL1, DaoReissigblowup}).  For the sake of clarity, we recall here the following sharp $(L^{m}\cap L^{2})-L^{2}$ linear estimates for (\ref{1.1}) which are recently used in \cite{Daosystem} (see again \cite{D'AbbiccoEbert2017}):  
\begin{Proposition}\cite{Daosystem} \label{Linearestimates0}
	Let $\delta\in(0, \frac{\sigma}{2})$ in (\ref{1.1}) and $m\in[1,2)$. The solutions $u$ to \eqref{1.1} satisfy the following $(L^{m}\cap L^{2})-L^{2}$ estimates
	\begin{align}
	\|\partial_{t}^{j}(-\Delta)^{a/2}u(t,\cdot)\|_{L^{2}(\mathbb{R}^{n})}&\lesssim(1+t)^{-\frac{n}{2(\sigma-\delta)}\left(\frac{1}{m}-\frac{1}{2}\right)-\frac{a+2j(\sigma-\delta)}{2(\sigma-\delta)}} \left\|u_{0}\right\| _{L^{m}(\mathbb{R}^{n})\cap H^{a+j\sigma}(\mathbb{R}^{n})}\nonumber \\
	&\hspace{-1.5cm} 
	+(1+t)^{-\frac{n}{2(\sigma-\delta)}\left(\frac{1}{m}-\frac{1}{2}\right)-\frac{a-2\delta}{2(\sigma-\delta)}-j} \left\|u_{1}\right\| _{L^{m}(\mathbb{R}^{n})\cap H^{\max\{0,a+(j-1)\sigma\}}(\mathbb{R}^{n})} \label{1.3}
	\end{align}
	for any $a\geq 0$, $j=0,1$ and $n>\frac{2m(2\delta-a)}{2-m}$. 
\end{Proposition}
The sharpness of these estimates causes by the fact that the structural damping $(-\Delta)^{\delta}\partial_{t}u$ with $2\delta\in(0,\sigma)$ leads the roots of the full characteristic equation to be real and the so-called \textit{diffusion phenomena} appears, for this reason the solution of (\ref{1.1}) inherits the sharpness of the estimates as those for diffusion equation. The application of the above linear estimates together with Duhamel's principle and Banach fixed point theorem gives us one of the most important quantities in recent years  that is the so-called \textit{critical exponent} for the corresponding semi-linear Cauchy problem  with power nonlinearity 
\begin{equation}\label{1.4}
\partial^{2}_{tt}u+(-\Delta)^{\sigma} u+(-\Delta)^{\delta}\partial_{t}u=|u|^{p}, \ u(0,x)=u_{0}(x), \ \ \partial_{t}u(0,x)=u_{1}(x), \ \ p>1.
\end{equation}
The critical exponent $p_{crit}$ can be seen as a threshold between the two behaviors of the solution, viz, a blow up (in finite time) for $1<p<p_{crit}$ or global existence (in time) for $p>p_{crit}$. This exponent is given by (see again \cite{D'AbbiccoEbert2017}, \cite{DaoReissigblowup})
\begin{equation}\label{1.5}
p_{crit}(n,m,\sigma, \delta)= 1+\frac{2m\sigma}{n-2m\delta}, \ \ \ \ \ \ m\in[1,2),
\end{equation}
where the initial data are chosen as follows : 
\begin{equation*}
u_{0} \in H^{\sigma}(\mathbb{R}^{n}) \cap L^{m}(\mathbb{R}^{n}),  \ \ u_{1} \in L^{2}(\mathbb{R}^{n})\cap L^{m}(\mathbb{R}^{n}).
\end{equation*}
From the above estimates we can see that the structural damping $(-\Delta)^{\delta}\partial_{t}u$ generates an additional time decay estimates of the form $(1+t)^{\frac{\delta}{\sigma-\delta}}$ with respect to $u_{1}$ and in all the $L^{2}$-norm of the solution and some of its derivatives. This motivate us to use the Pitt's inequality to neglect this loss of decay in the $L^{2}$ estimates for solutions to (\ref{1.1}). But this approach requires that the initial data $(u_{0}, u_{1})$ belong to some weighted spaces. Our main result is read as follows.
\begin{theorem} \label{Linearestimates1}
	Let $\sigma\geq 1$, $2\delta\in (0,\sigma)$ in (\ref{1.1}). Let $m\in(1,2)$ and $n>\frac{2m\delta}{m-1}$. The solutions $u$ to \eqref{1.1} satisfy the following $(L^{m}\cap L^{2}) -L^{2}$ estimates
	\begin{align}
	\|u(t,\cdot)\|_{L^{2}(\mathbb{R}^{n})}&\lesssim(1+t)^{-\frac{n}{2(\sigma-\delta)}\left(\frac{1}{m}-\frac{1}{2}\right) } \|u_{0}\|_{L^{m}(\mathbb{R}^{n})\cap L^{2}(\mathbb{R}^{n})}\nonumber \\
	&\hspace{4cm} 
	+(1+t)^{-\frac{n}{2(\sigma-\delta)}\left(\frac{1}{m}-\frac{1}{2}\right)}\|u_{1}\| _{L^{2\delta,m}(\mathbb{R}^{n})\cap L^{2}(\mathbb{R}^{n})}, \label{1.6}
	\end{align}
	\begin{align}
	\|(-\Delta)^{\sigma/2}u(t,\cdot)\|_{L^{2}(\mathbb{R}^{n})}&\lesssim(1+t)^{-\frac{n}{2(\sigma-\delta)}\left(\frac{1}{m}-\frac{1}{2}\right)-\frac{\sigma}{2(\sigma-\delta)}} \|u_{0}\|_{L^{m}(\mathbb{R}^{n})\cap H^{\sigma}(\mathbb{R}^{n})}\nonumber \\
	&\hspace{2cm} 
	+(1+t)^{-\frac{n}{2(\sigma-\delta)}\left(\frac{1}{m}-\frac{1}{2}\right)-\frac{\sigma}{2(\sigma-\delta)}}\|u_{1}\| _{L^{2\delta,m}(\mathbb{R}^{n})\cap         L^{2}(\mathbb{R}^{n})}, \label{1.7}
	\end{align}
	\begin{align}
	\|\partial_{t}u(t,\cdot)\|_{L^{2}(\mathbb{R}^{n})}&\lesssim(1+t)^{-\frac{n}{2(\sigma-\delta)}\left(\frac{1}{m}-\frac{1}{2}\right)-1} \|u_{0}\|_{L^{m}(\mathbb{R}^{n})\cap H^{\sigma}(\mathbb{R}^{n})}\nonumber \\
	&\hspace{4cm} 
	+(1+t)^{-\frac{n}{2(\sigma-\delta)}\left(\frac{1}{m}-\frac{1}{2}\right)-1}\|u_{1}\| _{L^{2\delta,m}(\mathbb{R}^{n})\cap L^{2}(\mathbb{R}^{n})}. \label{1.8}
	\end{align}
	In particular, the following $L^{2}-L^{2}$ estimates hold:
	\begin{align}
	\|u(t,\cdot)\|_{L^{2}(\mathbb{R}^{n})} &\lesssim \|u_{0}\|_{L^{2}(\mathbb{R}^{n})}+\|u_{1}\|_{L^{2\delta,2}(\mathbb{R}^{n})\cap L^{2}(\mathbb{R}^{n})}, \label{1.9}
	\\ 
	\|(-\Delta)^{\sigma/2}u(t,\cdot)\|_{L^{2}(\mathbb{R}^{n})} &\lesssim(1+t)^{-\frac{\sigma}{2(\sigma-\delta)}}\left(  \|u_{0}\|_{H^{\sigma}(\mathbb{R}^{n})}+\|u_{1}\|_{L^{2\delta,2}(\mathbb{R}^{n})\cap L^{2}(\mathbb{R}^{n})}\right), \label{1.10}
	\\
	\|\partial_{t}u(t,\cdot)\|_{L^{2}(\mathbb{R}^{n})} &\lesssim (1+t)^{-1}\left( \|u_{0}\|_{H^{\sigma}(\mathbb{R}^{n})}+\|u_{1}\|_{L^{2\delta,2}(\mathbb{R}^{n})\cap L^{2}(\mathbb{R}^{n})}\right).\label{1.11}
	\end{align}
\end{theorem}
\begin{remark}
	One of the nice influences of parabolic like structural damping $(-\Delta)^{\delta}\partial_{t} u$ is that the weight function $|x|^{2\delta}$ depends on the parameter $\delta$. In other words, when $u_{0} \in  H^{\sigma}(\mathbb{R}^{n})\cap L^{m}(\mathbb{R}^{n})$ and $u_{1} \in L^{2\delta,m}(\mathbb{R}^{n})\cap L^{2}(\mathbb{R}^{n})$ this hints us that the above decay estimates are really better than those in Proposition \ref{Linearestimates0} and are the same as those of the following diffusion equation, see for instance \cite{D'AbbiccoEbertdiffusion}:
	\begin{equation*}
	\partial_{t}v+(-\Delta)^{\sigma-\delta}v=0, \ v(0,x)=v_{0}(x),
	\end{equation*}
	with suitable initial data $v_{0}$.
\end{remark}
\begin{remark}
	The application of Pitt inequality gives us not only a better decay estimates, but also reduce the range of space dimension $n$ as follows $\frac{2m\delta}{m-1}<n$. This bound is less than $\frac{4m\delta}{2-m}$ in Proposition \ref{Linearestimates0} if and only if $m\in \left[ \frac{4}{3}, 2\right) $. Whereas if $m\in\left( 1, \frac{4}{3}\right] $ then $\frac{4m\delta}{2-m}\leq \frac{2m\delta}{m-1}$.
\end{remark}
\begin{remark}
	We have seen that the admissible data spaces in (\ref{1.6})-(\ref{1.8}) are taken as follows: 
	$$u_{0} \in H^{\sigma}(\mathbb{R}^{n})\cap L^{m}(\mathbb{R}^{n}), \ \ u_{1} \in L^{2}(\mathbb{R}^{n})\cap L^{2\delta,m}(\mathbb{R}^{n}), \ \ m\in(1,2].$$
	But, we can also obtain another better decay estimates with respect to $u_{0}$ if we choose the initial data as
	$$u_{0} \in H^{\sigma}(\mathbb{R}^{n})\cap L^{2\delta,m}(\mathbb{R}^{n}), \ \ u_{1} \in L^{2}(\mathbb{R}^{n})\cap L^{2\delta,m}(\mathbb{R}^{n}), \ \ m\in(1,2],$$
	which motivates us to write the following corollary.
\end{remark}
\begin{corollary}
	Let $\sigma\geq 1$, $2\delta\in (0,\sigma)$ in (\ref{1.1}). Let $m\in(1,2)$ and $n>\frac{2m\delta}{m-1}$. The solutions $u$ to \eqref{1.1} satisfy the following $(L^{m}\cap L^{2}) -L^{2}$ estimates
	\begin{align*}
	\|u(t,\cdot)\|_{L^{2}(\mathbb{R}^{n})}&\lesssim(1+t)^{-\frac{n}{2(\sigma-\delta)}\left(\frac{1}{m}-\frac{1}{2}\right)-\frac{\delta}{\sigma-\delta}} \|u_{0}\|_{L^{2\delta,m}(\mathbb{R}^{n})\cap L^{2}(\mathbb{R}^{n})}\nonumber \\
	&\hspace{4cm} 
	+(1+t)^{-\frac{n}{2(\sigma-\delta)}\left(\frac{1}{m}-\frac{1}{2}\right)}\|u_{1}\| _{L^{2\delta,m}(\mathbb{R}^{n})\cap L^{2}(\mathbb{R}^{n})},
	\end{align*}
	\begin{align*}
	\|(-\Delta)^{\sigma/2}u(t,\cdot)\|_{L^{2}(\mathbb{R}^{n})}&\lesssim(1+t)^{-\frac{n}{2(\sigma-\delta)}\left(\frac{1}{m}-\frac{1}{2}\right)-\frac{\sigma+2\delta}{2(\sigma-\delta)}} \|u_{0}\|_{L^{2\delta,m}(\mathbb{R}^{n})\cap H^{\sigma}(\mathbb{R}^{n})}\nonumber \\
	&\hspace{2cm} 
	+(1+t)^{-\frac{n}{2(\sigma-\delta)}\left(\frac{1}{m}-\frac{1}{2}\right)-\frac{\sigma}{2(\sigma-\delta)}}\|u_{1}\| _{L^{2\delta,m}(\mathbb{R}^{n})\cap         L^{2}(\mathbb{R}^{n})},
	\end{align*}
	\begin{align*}
	\|\partial_{t}u(t,\cdot)\|_{L^{2}(\mathbb{R}^{n})}&\lesssim(1+t)^{-\frac{n}{2(\sigma-\delta)}\left(\frac{1}{m}-\frac{1}{2}\right)-1-\frac{\delta}{\sigma-\delta}} \|u_{0}\|_{L^{2\delta,m}(\mathbb{R}^{n})\cap H^{\sigma}(\mathbb{R}^{n})}\nonumber \\
	&\hspace{4cm} 
	+(1+t)^{-\frac{n}{2(\sigma-\delta)}\left(\frac{1}{m}-\frac{1}{2}\right)-1}\|u_{1}\| _{L^{2\delta,m}(\mathbb{R}^{n})\cap L^{2}(\mathbb{R}^{n})}.
	\end{align*}
	In particular, the following $L^{2}-L^{2}$ estimates hold:
	\begin{align*}
	\|u(t,\cdot)\|_{L^{2}(\mathbb{R}^{n})} &\lesssim (1+t)^{-\frac{\delta}{\sigma-\delta}}\|u_{0}\|_{L^{2\delta,2}(\mathbb{R}^{n})\cap L^{2}(\mathbb{R}^{n})}+\|u_{1}\|_{L^{2\delta,2}(\mathbb{R}^{n})\cap L^{2}(\mathbb{R}^{n})},
	\\ 
	\|(-\Delta)^{\sigma/2}u(t,\cdot)\|_{L^{2}(\mathbb{R}^{n})} &\lesssim(1+t)^{-\frac{\sigma+2\delta}{2(\sigma-\delta)}}  \|u_{0}\|_{L^{2\delta,2}(\mathbb{R}^{n})\cap H^{\sigma}(\mathbb{R}^{n})}+(1+t)^{-\frac{\sigma}{2(\sigma-\delta)}}\|u_{1}\|_{L^{2\delta,2}(\mathbb{R}^{n})\cap L^{2}(\mathbb{R}^{n})},
	\\
	\|\partial_{t}u(t,\cdot)\|_{L^{2}(\mathbb{R}^{n})} &\lesssim(1+t)^{-1-\frac{\delta}{\sigma-\delta}}  \|u_{0}\|_{L^{2\delta,2}(\mathbb{R}^{n})\cap H^{\sigma}(\mathbb{R}^{n})}+(1+t)^{-1}\|u_{1}\|_{L^{2\delta,2}(\mathbb{R}^{n})\cap L^{2}(\mathbb{R}^{n})}.
	\end{align*}
\end{corollary}
\begin{remark}
	We can also use the Pitt inequality to obtain a better $(L^{m}\cap L^{2}) -L^{2}$ estimates for the solution and some of its derivatives to the homogeneous linear $\sigma$-evolution equations with structural damping, where $\delta \in \left(\frac{\sigma}{2}, \sigma \right)$. In this case, the space dimension $n$ will satisfies the condition $n>\frac{m\sigma}{m-1}$ instead of $n>\frac{2m\sigma}{2-m}$, see for instance \cite{PhamKainaneReissig}, \cite{Daosystem}. Moreover, the initial data $(u_{0}, u_{1})$ are chosen as follows:
	$$u_{0} \in H^{2\delta}(\mathbb{R}^{n})\cap L^{\sigma,m}(\mathbb{R}^{n}), \ \ u_{1} \in L^{2}(\mathbb{R}^{n})\cap L^{\sigma,m}(\mathbb{R}^{n}), \ \ m\in(1,2].$$
\end{remark}
The remaining part of this paper is organized as follows: In Section \ref{Preliminaries} we recall two important tools. Section \ref{proof} is devoted to prove Theorem \ref{Linearestimates1} and we closed it by mentioning some of our future studies. 
\section{ Preliminaries}\label{Preliminaries} 
The following Pitt inequality is very important and it is a generalization of the well-known Hausdorff-Young inequality. It reads as follows.
\begin{lemma}\cite{BenedettoHeinig, DeCarliGorbachevTikhonov}\label{Pitt's inequality}
	Let $1<r_{2}\leq r_{1}<\infty$, $s_{1}\geq 0$, $s_{2}\geq 0$ and $n\geq 1$. There is $c>0$ such that  
	\begin{equation}\label{2.1}
	\||\xi|^{-s_{1}}\mathcal{F}(f)\|_{L^{r_{1}}(\mathbb{R}^{n})}\leq c \||x|^{s_{2}}f\|_{L^{r_{2}}(\mathbb{R}^{n})} \ \ \text{for all} \ f\in \mathcal{S}(\mathbb{R}^{n}). 
	\end{equation}
	where 
	$$r_{1}s_{1}<n,\   \ \ \frac{r_{2}s_{2}}{r_{2}-1}<n , \ \ s_{1}=s_{2}+n\left(\frac{1}{r_{1}}+\frac{1}{r_{2}}-1\right).$$ 
\end{lemma}
\begin{remark}
	When $s_{1}=s_{2}=0$ and $\frac{1}{r_{1}}+\frac{1}{r_{2}}=1$, the Pitt inequality becomes the standard Hausdorff-Young inequality:
	\begin{equation}\label{2.2}
	\|\mathcal{F}(f)\|_{L^{r_{1}}(\mathbb{R}^{n})}\lesssim \|f\|_{L^{r_{2}}(\mathbb{R}^{n})} \ \ \text{for all} \ f\in \mathcal{S}(\mathbb{R}^{n}).
	\end{equation}
\end{remark}
A direct application of H\"{o}lder's inequality gives.
\begin{lemma}\cite{IkehataTakeda}\label{Productinequality}
	Let $n\geq 1$, $1\leq m \leq 2$ and $\frac{1}{m}+\frac{1}{m'}=1$. Then it holds that
	\begin{equation}
	\|fg\|_{L^{2}(\mathbb{R}^{n})}\leq \|f\|_{L^{\frac{2m}{2-m}}(\mathbb{R}^{n})}\|g\|_{L^{m'}(\mathbb{R}^{n})}.
	\end{equation}
\end{lemma}
\begin{proof}
	We have by the H\"{o}lder's inequality $$\|fg\|^{2}_{L^{2}(\mathbb{R}^{n})}=\|f^{2}g^{2}\|_{L^{1}(\mathbb{R}^{n})}\leq\|f^{2}\|_{L^{q}(\mathbb{R}^{n})}\|g^{2}\|_{L^{m'/2}(\mathbb{R}^{n})}$$
	provided that $$\frac{1}{q}+\frac{1}{\frac{m'}{2}}=1\ \ \ i.e, \ \ q=\frac{m}{2-m}.$$
	Hence, we deduce 
	$$\|fg\|_{L^{2}(\mathbb{R}^{n})}\leq \|f\|_{L^{2q}(\mathbb{R}^{n})}\|g\|_{L^{m'}(\mathbb{R}^{n})}.$$
\end{proof}
Now, we have all the basic tools to prove our result.
\section{Proof of Theorem \ref{Linearestimates1}}\label{proof}
First, we need to know the formal representation of solutions. By
applying the partial Fourier transform $\mathcal{F}u(t,x):=\hat{u}(t,\xi)$ to (\ref{1.1}) we obtain the following second order differential equation :
\begin{equation*}
\partial^{2}_{t}\hat{u}+|\xi|^{2\delta}\partial_{t}\hat{u}+|\xi|^{2\sigma}\hat{u}=0,\  
\hat{u}(0,\xi)=\hat{u}_{0}(\xi), \ \partial_{t}\hat{u}(0,\xi)=\hat{u}_{1}(\xi).
\end{equation*}
A simple calculation leads to the formal solution as follows:
\begin{equation}\label{3.1}
\partial_{t}^{i}|\xi|^{a}\hat{u}(t,\xi)=|\xi|^{a}\frac{\lambda_{1}\lambda_{2}^{i}e^{\lambda_{2}t}-\lambda_{2}\lambda_{1}^{i}e^{\lambda_{1}t}}{\lambda_{1}-\lambda_{2}}\hat{u}_{0}(\xi)+|\xi|^{a}\frac{\lambda_{1}^{i}e^{\lambda_{1}t}-\lambda_{2}^{i}e^{\lambda_{2}t}}{\lambda_{1}-\lambda_{2}}\hat{u}_{1}(\xi)
\end{equation}
where $i \in \{0,1\}$, $a\geq 0$. The roots $\lambda_{1}$, $\lambda_{2}$ of the characteristic equation are given by 
$$\lambda_{1,2}:=\lambda_{1,2}(\xi)=\frac{-|\xi|^{2\delta}\pm \sqrt{|\xi|^{4\delta}-4|\xi|^{2\sigma}}}{2}$$ and behave like 
\begin{equation}\label{3.2}
\left\lbrace \begin{matrix}
\lambda_{1}\approx -|\xi|^{2(\sigma-\delta)},\  \lambda_{2}\approx-|\xi|^{2\delta},\  \lambda_{1}-\lambda_{2}\approx|\xi|^{2\delta} \hfill& { as \ \ |\xi| \to 0},
&\cr 
\\
\Re\lambda_{1,2}=-\frac{|\xi|^{2\delta}}{2},\  \Im\lambda_{1,2}=\pm|\xi|^{\sigma},\ \ \Im(\lambda_{1}-\lambda_{2})\approx  |\xi|^{\sigma}\hfill & { as \ \ |\xi| \to \infty}.
\end{matrix}\right.
\end{equation}  
\begin{proof}
	In order to prove (\ref{1.6})-(\ref{1.8}), we derive $L^{m}-L^{2}$ estimates for small frequencies and $L^{2}-L^{2}$ estimates for large frequencies. Here, we note that the linear estimates for small frequencies give us the type of decay rate which is polynomial, while the linear estimates for small frequencies give us an exponential decay together with a suitable regularities of initial data. This explains why we use the following decomposition of the formal solution (\ref{3.1}):
	$$\partial_{t}^{i}|\xi|^{a}\hat{u}(t,\xi) =\partial_{t}^{i}|\xi|^{a}\chi(\xi)\hat{u}(t,\xi)+\partial_{t}^{i}|\xi|^{a}(1-\chi(\xi))\hat{u}(t,\xi),$$
	where $\chi$ is a smooth cut-off function satisfies $\chi(\xi)=1$ when $|\xi|\to 0$ and $1-\chi(\xi)=1$ when $|\xi|\to \infty$. Using the asymptotic behavior of the characteristics roots (\ref{3.2}) we may estimates the low frequency part of the solution as follows :
	\begin{align*}
	\|\partial_{t}^{i}|\xi|^{a}\chi(\xi)\hat{u}(t,\xi)\|_{L^{2}}&\lesssim\left\| |\xi|^{a}\frac{\lambda_{1}\lambda_{2}^{i}e^{\lambda_{2}t}-\lambda_{2}\lambda_{1}^{i}e^{\lambda_{1}t}}{\lambda_{1}-\lambda_{2}}\chi(\xi)\hat{u}_{0}(\xi)\right\|_{L^{2}}\nonumber\\
	&\hspace{4cm}
	+\left\||\xi|^{a}\frac{\lambda_{1}^{i}e^{\lambda_{1}t}-\lambda_{2}^{i}e^{\lambda_{2}t}}{\lambda_{1}-\lambda_{2}}\chi(\xi)\hat{u}_{1}(\xi)\right\|_{L^{2}}\nonumber \\
	&\hspace{-2cm} 
	\lesssim \left\||\xi|^{a+2(\sigma+\delta(i-2))}e^{-|\xi|^{2\delta}t}\chi(\xi)\hat{u}_{0}(\xi)\right\|_{L^{2}}+\left\||\xi|^{a+2i(\sigma-\delta)}e^{-|\xi|^{2(\sigma-\delta)}t}\chi(\xi)\hat{u}_{0}(\xi)\right\|_{L^{2}}\nonumber
	\\
	&\hspace{-2cm} 
	+\left\||\xi|^{a+2i(\sigma-\delta)}e^{-|\xi|^{2(\sigma-\delta)}t}|\xi|^{-2\delta}\chi(\xi)\hat{u}_{1}(\xi)\right\|_{L^{2}}+\left\||\xi|^{a+2i\delta}e^{-|\xi|^{2\delta}t}|\xi|^{-2\delta}\chi(\xi)\hat{u}_{1}(\xi)\right\|_{L^{2}}.\nonumber
	\end{align*} 
	Now, we use the Hausdorff-Young inequality from (\ref{2.2}) we have 
	$$
	\left\||\xi|^{a+2(\sigma+\delta(i-2))}e^{-|\xi|^{2\delta}t}\chi(\xi)\hat{u}_{0}(\xi)\right\|_{L^{2}}\lesssim \left\| |\xi|^{a+2(\sigma+\delta(i-2))}e^{-|\xi|^{2\delta}t}\chi(\xi)\right\| _{L^{\frac{2m}{2-m}}}\|\hat{u}_{0}(\xi)\|_{L^{m'}}. 
	$$
	To estimate the middle norm we proceed as follows:
	$$\left\| |\xi|^{a+2(\sigma+\delta(i-2))}e^{-|\xi|^{2\delta}t}\chi(\xi)\right\| _{L^{\frac{2m}{2-m}}}=\left\| |\xi|^{a+2(\sigma+\delta(i-2))}e^{-(1+t)|\xi|^{2\delta}}e^{|\xi|^{2\delta}}\chi(\xi)\right\|_{L^{\frac{2m}{2-m}}}$$
	$$
	\lesssim \left\| |\xi|^{a+2(\sigma+\delta(i-2))}e^{-(1+t)|\xi|^{2\delta}}\chi(\xi)\right\| _{L^{\frac{2m}{2-m}}}.
	$$
	We have for $R=|\xi|$ 
	$$\left\| |\xi|^{a+2(\sigma+\delta(i-2))}e^{-(1+t)|\xi|^{2\delta}}\chi(\xi)\right\|^{\frac{2m}{2-m}} _{L^{\frac{2m}{2-m}}}=\int_{\mathbb{R}^{n}}|\xi|^{\left(a+2(\sigma+\delta(i-2))\right) \frac{2m}{2-m}}e^{-\frac{2m}{2-m}(1+t)|\xi|^{2\delta}}d\xi$$
	$$\lesssim \int_{0}^{\infty}R^{\left(a+2(\sigma+\delta(i-2))\right) \frac{2m}{2-m}+n-1}e^{-\frac{2m}{2-m}(1+t)R^{2\delta}}dR.$$
	By taking the change of variable $z=(\frac{2m}{2-m}(1+t))^{1/2\delta}R$ we obtain
	$$ \int_{0}^{\infty}R^{\left(a+2(\sigma+\delta(i-2))\right) \frac{2m}{2-m}+n-1}e^{-\frac{2m}{2-m}(1+t)R^{2\delta}}dR\lesssim(1+t)^{-\frac{n}{2\delta}\left(\frac{1}{m}-\frac{1}{2}\right)-\frac{a+2(\sigma+\delta(i-2))}{2\delta}}.$$
	Hence we conclude
	$$\left\||\xi|^{a+2(\sigma+\delta(i-2))}e^{-|\xi|^{2\delta}t}\chi(\xi)\hat{u}_{0}(\xi)\right\|_{L^{2}}\lesssim (1+t)^{-\frac{n}{2\delta}\left(\frac{1}{m}-\frac{1}{2}\right)-\frac{a+2(\sigma+\delta(i-2))}{2\delta}}\|u_{0}\|_{L^{m}},$$  
	similarly we can also prove	
	$$
	\left\||\xi|^{a+2i(\sigma-\delta)}e^{-|\xi|^{2(\sigma-\delta)}t}\hat{u}_{0}(\xi)\right\|_{L^{2}}\lesssim (1+t)^{-\frac{n}{2(\sigma-\delta)}\left(\frac{1}{m}-\frac{1}{2}\right)-\frac{a+2i(\sigma-\delta)}{2(\sigma-\delta)}}\|u_{0}\|_{L^{m}}.
	$$
	For the remaining norms with respect to $u_{1}$, we use the Pitt inequality instead of the Hausdorff-Young inequality, this is the crucial point	
	$$
	\left\||\xi|^{a+2i(\sigma-\delta)}e^{-|\xi|^{2(\sigma-\delta)}t}|\xi|^{-2\delta}\hat{u}_{1}(\xi)\right\|_{L^{2}} \lesssim \left\| |\xi|^{a+2i(\sigma-\delta)}e^{-|\xi|^{2(\sigma-\delta)}t}\right\| _{L^{\frac{2m}{2-m}}}\||\xi|^{-2\delta}\hat{u}_{1}(\xi)\|_{L^{m'}}.
	$$
	$$
	\lesssim(1+t)^{-\frac{n}{2(\sigma-\delta)}\left(\frac{1}{m}-\frac{1}{2}\right)-\frac{a+2i(\sigma-\delta)}{2(\sigma-\delta)}}\||x|^{2\delta}u_{1}\|_{L^{m}},
	$$
	and
	$$
	\left\||\xi|^{a+2i\delta}e^{-|\xi|^{2\delta}t}|\xi|^{-2\delta}\hat{u}_{1}(\xi)\right\|_{L^{2}}\lesssim (1+t)^{-\frac{n}{2\delta}\left(\frac{1}{m}-\frac{1}{2}\right)-\frac{a+2i\delta}{2\delta}}\||x|^{2\delta}u_{1}\|_{L^{m}}.
	$$
	Here the above last estimates hold if and only if from the Pitt inequality: $$\frac{2\delta m}{m-1}<n \ \ and\ \  \frac{1}{m}+\frac{1}{m'}, \ \ m\in(1,2).$$
	Finally, by the formula of Parseval-Plancherel we can conclude the desired estimates for small frequencies
	$$
	\|\partial_{t}^{i}(-\Delta)^{a/2}\chi(\cdot)u(t,\cdot)\|_{L^{2}}\lesssim(1+t)^{-\frac{1}{2(\sigma-\delta)}\left(\frac{1}{m}-\frac{1}{2}\right)-\frac{a+2i(\sigma-\delta)}{2(\sigma-\delta)}}\left(\|u_{0}\|_{L^{m}}+ \||x|^{2\delta}u_{1}\|_{L^{m}}\right). 
	$$
	For large frequencies, we can obtain the same $L^{2}-L^{2}$ estimates as in \cite{Daosystem}:
	$$\|\partial_{t}^{i}(-\Delta)^{a/2}(1-\chi(\cdot))u(t,\cdot)\|_{L^{2}}\lesssim e^{-t}\left(  \|u_{0}\|_{H^{a+i\sigma}}+\|u_{1}\|_{H^{\max\{a+(i-1)\sigma,0\}}}\right) .$$
	For the $L^{2}-L^{2}$ linear estimates in small frequencies we can proceed as above, this concludes the proof of Theorem \ref{Linearestimates1}.  		
\end{proof}
\subsection{Future applications}
We strongly believe that when the linear estimates become better, the critical exponent is improved. So, in future work, it is interesting to apply the Theorem \ref{Linearestimates1} and study the corresponding semi-linear Cauchy problem
\begin{equation}
\partial^{2}_{t}u+(-\Delta)^{\sigma} u+(-\Delta)^{\delta}\partial_{t}u=f(t,u), \ u(0,x)=u_{0}(x), \ \ \partial_{t}u(0,x)=u_{1}(x),
\end{equation}
where $$f(t,u)=|u|^{p}, \ \  f(t,u)=\int_{0}^{t}(t-s)^{-\gamma}|u|^{p}ds, \ \ \ \ p>1,  \ \ \gamma\in(0,1).$$
The fact that the initial data $u_{1}$ lives in weighted space this motivate us to look for the application of the classical Caffarelli-Kohn-Nirenberg inequality proved in \cite{CaffarelliKohnNirenberg} which is a generalization of the well-know Gagliardo-Nirenberg inequality.


\vskip+0.5cm

\noindent \textbf{Author address:}

\vskip+0.3cm

\noindent \textbf{Khaldi Said}

Laboratory of Analysis and Control of PDEs, Djillali Liabes University,  P.O. Box 89, Sidi-Bel-Abb\`{e}s 22000, Algeria

{\itshape E-mails:} \texttt{saidookhaldi@gmail.com, said.khaldi@univ-sba.dz}

\end{document}